\documentclass[11pt,a4paper,twoside]{article}
\usepackage{amsfonts,amsmath,amssymb,graphicx}

\usepackage{setspace}

\title{The multivariate Serre conjecture ring
}
\author{Luc Guyot(\thanks{~ email: luc.guyot.ge@gmail.com}~),  Ihsen
Yengui (\thanks{~ D\'epartement de Math\'ematiques,  Facult\'e des
Sciences, Universit\'e de Sfax, 3000 Sfax, Tunisia, email:
ihsen.yengui@fss.rnu.tn}~)   }

\DeclareSymbolFont{lasy}{U}{lasy}{m}{n}
\SetSymbolFont{lasy}{bold}{U}{lasy}{b}{n}
\let\Box\undefined
\DeclareMathSymbol\Box{\mathord}{lasy}{"32}

\newtheorem{theorem}{Theorem}
\newtheorem{proposition}[theorem]{Proposition}
\newtheorem{lemma}[theorem]{Lemma}
\newtheorem{corollary}[theorem]{Corollary}

\newtheorem{example}[theorem]{Example}

\newtheorem{definition}[theorem]{Definition}
\newtheorem{notation}[theorem]{Notation}

\newcommand {\junk}[1]{}

\newenvironment{proof}[1]{
\trivlist \item[\hskip \labelsep{\bf #1}]}{\hfill\mbox{$\Box$}
\endtrivlist}
\makeatletter
\def\maketitle{\par
\begingroup
\def\@makefnmark{\hbox
to 0pt{$^{\@thefnmark}$\hss}} \if@twocolumn
\twocolumn[\@maketitle] \else \newpage \global\@topnum\z@
\@maketitle \fi\thispagestyle{plain}\@thanks
\endgroup
\setcounter{footnote}{0}
\let\maketitle\relax
\let\@maketitle\relax
\gdef\@thanks{}\gdef\@author{}\gdef\@title{}\let\thanks\relax}
\makeatother

\def\.@{\char'76}

\def\.@{\char'76}

\DeclareMathAlphabet{\cmttl}{OT1}{cmtt}{n}{sl}
\newcommand{\tlex}{\cmttl{lex}}
\newcommand{\mdeg}{{\rm mdeg} }

\newcommand{\LM}{\operatorname{LM}}
\newcommand{\LC}{\operatorname{LC}}
\newcommand{\LT}{\operatorname{LT}}
\newcommand{\TM}{\operatorname{TM}}
\newcommand{\TC}{\operatorname{TC}}
\newcommand{\LMl}{\operatorname{LM}_{\tlex}}
\newcommand{\LCl}{\operatorname{LC}_{\tlex}}
\newcommand{\LTl}{\operatorname{LT}_{\tlex}}

\newcommand{\TCl}{\operatorname{TC}_{\tlex}}
\newcommand{\TT}{\operatorname{TT}}

\newcommand{\Id}{{\operatorname{Id}}}

\newcommand{\sotq}[2]{\{\,#1\mathrel;\allowbreak#2\,\}}

\newcommand{\ord}{\prec}
\newcommand{\lcm}{\operatorname{lcm}}
\newcommand{\Pol}{\mathbf{R}[X_1, \dots, X_n]}
\newcommand{\PolY}{\mathbf{A}[Y_1, \dots, Y_p]}
\newcommand{\Polk}{\mathbf{R}[X_1^k, \dots, X_n^k]}
\newcommand{\LPol}{\mathbf{R}[X_1^{\pm 1}, \dots, X_n^{\pm 1}]}
\newcommand{\SCR}{\mathbf{R}_{\prec}\langle X_1, \dots, X_n\rangle}

\newcommand{\SCRl}{\mathbf{R}_{\tlex}\langle X_1, \dots, X_n\rangle}
\newcommand{\RX}{\mathbf{R}[X]}

\newcommand{\SRX}{\mathbf{R} \langle X \rangle}

\newcommand{\Rk}{\mathcal{R}_k}
\newcommand{\MatZ}{\mathbb{Z}^{n \times n}}
\newcommand{\MatZp}{\mathbb{Z}_{\ge 0}^{n \times n}}

\def \n{\noindent}

\def \RR{\mathbb{R}}
\def \QQ{\mathbb{Q}}
\def \NN{\mathbb{N}}
\def \ZZ{\mathbb{Z}}

\def \mod {{\rm mod}}

\def\gen#1{\langle{#1}\rangle}

\def \A{{\bf A}}
\def \B{{\bf B}}

\def \V{{\bf V}}
\def \R{{\bf R}}
\def \K{{\bf K}}

\def \B{{\bf B}}

\def \noi {\noindent}
\def \ss {\smallskip}
\def \sni {\ss\noi}

\def \bs {\bigskip}
\def \bni {\bs\noi}
\newcommand{\upto}{,\ldots,}

\begin{document}

\maketitle

\begin{abstract}
It is well-known that for any commutative unitary ring $\R$,
the Serre conjecture ring $\SRX$, i.e., the localization of the univariate polynomial ring $\RX$ at monic polynomials,
is a B\'ezout domain of Krull dimension $\leq 1$  if  so is  $\R$. Consequently, defining by induction
$\R\gen{X_1,\ldots,X_n} := (\R\gen{X_1,\ldots,X_{n-1}}) \gen{X_n}$,
the ring $\R\gen{X_1,\ldots,X_n}$  is a B\'ezout domain of Krull dimension $\leq 1$ if so is $\R$.
The fact that $\R\gen{X_1,\ldots,X_n}$ is a B\'ezout domain when $\R$ is a valuation domain of Krull dimension $\leq 1$ was
the cornerstone of Brewer and Costa's theorem stating that if $\R$ is a one-dimensional arithmetical ring
then finitely generated projective $\Pol$-modules are extended.
It is also the key of the proof of the  Gr\"obner Ring Conjecture in the lexicographic order case,
namely the fact that for any valuation domain $\R$ of Krull dimension $\leq 1$,
any $n \in \NN_{>0}$, and any finitely generated ideal $I$ of  $\R[X_1,\ldots, X_n]$,
the ideal $\LT(I )$ generated by the
leading terms of the elements of $I$ with respect to the lexicographic monomial order is finitely generated.
Since the ring $\R\gen{X_1,\ldots,X_n}$ can also be   defined directly as the localization of the multivariate polynomial ring $\Pol$
at polynomials whose leading coefficients according to the lexicographic monomial order with
$X_1 < X_2 < \cdots < X_n$ is $1$, we propose to  generalize the fact that $\R\gen{X_1,\ldots,X_n}$ is
a B\'ezout domain of Krull dimension $\leq 1$  if so is $\R$ to any rational monomial
order, bolstering the evidence for the Gr\"obner Ring Conjecture in the rational case.
We give an example showing that this result is no more true in the irrational case.
\end{abstract}

\bni MSC 2020: Primary 13B25, Secondary 13B30, 13F05

\sni Key words: rational monomial order, irrational monomial order, leading terms ideal, valuation domains, Gr\"obner ring conjecture.

\section*{Acknowledgement}

The second author thanks the Alexander von Humboldt Foundation
for funding his stay at the Technische Universit\"at M\"unchen,
during which the research for this paper was done.
He also thanks Gregor Kemper for fruitful and interesting conversations.

\section*{Introduction} \label{sec Introduction}

This paper is written in the framework of Bishop style constructive mathematics (see~\cite{Bi67,LQ,MRR,Y5}).
All the considered rings are supposed to be  commutative, unitary,
and strongly discrete (i.e., equipped with a membership test to finitely generated ideals).

\medskip

\n We will be working with monomial orders on the polynomial ring
$\Pol$ over a ring $\R$. According to \cite[Theorem~1.2]{Kemper:Trun:VanAnh:2017},
every monomial preorder~${\ord}$ is given by a matrix
$M \in \mathbb{R}^{m \times n}$ for some $m > 0$ in the following way:

\[
  X_1^{e_1} X_2^{e_2} \cdots X_n^{e_n} {\ord} \, X_1^{f_1} X_2^{f_2} \cdots
  X_n^{f_n} \qquad \Longleftrightarrow \qquad M \cdot
  \begin{pmatrix} e_1 \\ \vdots \\ e_n \end{pmatrix} <_{\tlex} M \cdot
  \begin{pmatrix} f_1 \\ \vdots \\ f_n \end{pmatrix}
\]
($e_i,f_i \in \mathbb{N}$), where $\tlex$ denotes the lexicographic order.
A given matrix $M$ defines a monomial preorder if and only if all
columns are nonzero and their first nonzero entry is positive. Notice
that different matrices may define the same monomial preorder. For
example, adding a multiple of a row of $M$ to a lower row does not
change the preorder. By doing this repeatedly, we can replace $M$ by a matrix
with nonnegative entries, so from now on we will assume that
$M \in \RR_{\ge 0}^{m \times n}$. Let us call a preorder {\em rational} if it can be defined by a matrix with rational entries,
which can then be assumed to be nonnegative integers. All monomial
orders used in practice are rational. A rational preorder is a
monomial order (i.e., there are no ties between monomials) if and only
if $M$ has rank~$n$, so in this case we may assume $m = n$. By
contrast, irrational preorders can be orders even if $m < n$, for
example if $M$ consists of a single row of real numbers that are
linearly independent over $\QQ$. A monomial preorder is said to be {\em graded}
if the first row of $M$ has only positive entries. Graded rational monomial orders were used
in \cite{KY} to give an effective characterization of the valuative dimension.

\medskip

\n {\bf The Gr\"obner ring conjecture (the updated version, \cite[page 535]{Y-JA21}) :}

\medskip

\n A valuation domain $\R$ has Krull dimension $\leq 1$ if and only if for any $n \in \NN_{>0}$,
for any rational monomial order ${\ord}$ on $\Pol$, and   any finitely generated ideal $I$ of  $\Pol$,
the ideal $\LT_{\ord}(I)$ generated by the leading terms of the elements of $I$ with respect to $\ord$ is finitely generated.

\medskip

It is well-known that for any ring $\R$, the Serre conjecture ring $\SRX$, i.e., the localization of the univariate polynomial ring $
\RX$ at monic polynomials,   is a B\'ezout domain of Krull dimension $\leq 1$  if  so is  $\R$.
Consequently, defining by induction $\R\gen{X_1,\ldots,X_n} := (\R\gen{X_1,\ldots, X_{n-1}})\gen{X_n}$,
the ring $\R\gen{X_1,\ldots,X_n}$  is a B\'ezout domain of Krull dimension $\leq 1$  if so is  $\R$.
The fact that $\R\gen{X_1,\ldots,X_n}$  is a B\'ezout domain when   $\R$ is a valuation domain of Krull dimension $\leq 1$ was
 the cornerstone of  Brewer and Costa theorem \cite[Corollary 3]{BC}  stating that if $\R$ is
a one-dimensional arithmetical ring then finitely generated projective $\Pol$-modules are extended on the one hand,
and on the other, the key of the proof of the  Gr\"obner Ring Conjecture in the lexicographic order case \cite[Theorem 4]{Y4}.
Since the ring $\R\gen{X_1,\ldots,X_n}$ can also be   defined directly as the localization of
the multivariate polynomial ring $\Pol$ at polynomials whose leading coefficients according  to the lexicographic monomial order with
$X_1 < X_2 < \cdots < X_n$ is $1$, we propose to generalize the aforementioned  result to any rational monomial
order, bolstering the evidence for the Gr\"obner Ring Conjecture in the rational case.
We show with Example \ref{CounterExample} below that this result is no more true in the irrational case.

\section{The leading terms ideal over a valuation domain}\label{s1}

\subsection{A Reminder}\label{s111}

\begin{definition}
\label{def10100} ~

{\rm Let $\R$ be a ring, consider a monomial order ${\ord}$ on $\Pol$, and let $f=\sum_{i=0}^m a_{i} X^{\alpha_i}$
be a nonzero polynomial in $\Pol$ with
$X^{\alpha_0} {\ord} \cdots {\ord}\, X^{\alpha_m}$ and $a_{i} \in \R \setminus \{ 0 \}$.

 \begin{description}
\item [{\rm (1)}] The $X^{\alpha_i}$ (resp. the $a_{_i}X^{\alpha_i}$) are
 called the {\it monomials} (resp. the {\it terms}) of $f$.

\item [{\rm (2)}] The {\it multidegree} of $f$ is $\mdeg_{\ord}(f):=\alpha_m \in \NN^n$.

\item [{\rm (3)}]  The {\it leading coefficient} of $f$ is $\LC_{\ord}(f) := a_{m} \in \R\setminus \{ 0 \}$.

\item [{\rm (4)}]  The multivariate polynomial $f$ is said to be ${\ord}$-monic if $\LC_{\ord}(f)=1$.

\item [{\rm (5)}] The {\it leading monomial} of $f$ is $\LM_{\ord}(f) := X^{\alpha_m}$.

\item [{\rm (6)}] The {\it leading term} of $f$ is $\LT_{\ord}(f) := \LC_{\ord}(f) \, \LM_{\ord}(f)$.

\item [{\rm (7)}] The {\it leading terms}  of a  set $F\subseteq \Pol$ is $\LT_{\ord}(F) := \{\LT_{\ord}(f);\;f \in F\setminus \{ 0 \}  \}$.

\item [{\rm (8)}] The {\it leading terms ideal}  of a nonzero finitely generated ideal $I$ of
 $\R[X_1,\ldots,
X_n]$ is $\LT_{\ord}(I) := \gen{\LT_{\ord}(f);\; f \in I\setminus \{ 0 \}}$ (an ideal of $\Pol$).

\item [{\rm (9)}] The {\it leading coefficient ideal}  of a nonzero finitely generated ideal $I$ of
 $\Pol$ is $\LC_{\ord}(I) := \gen{\LC_{\ord}(f);\; f \in I\setminus \{ 0 \}}$ (an ideal of $\R$)
\end{description}

\n It is sometimes convenient to extend a monomial order to the monomials in
the Laurent polynomial ring $\R[X_1^{\pm 1} \upto X_n^{\pm
  1}]$. If~${\ord}$ is a monomial ordering and
$f \in \R[X_1^{\pm 1} \upto X_n^{\pm 1}]$ is a nonzero Laurent
polynomial, we write $\LM_{\ord}(f)$ for its leading monomial,
$\LC_{\ord}(f)$ for the coefficient of this monomial, and $\LT_{\ord}(f):=\LC_{\ord}(f) \, \LM_{\ord}(f)$.
Likewise, we define the \emph{trailing coefficient} $\TC_{\ord}(f)$, the \emph{trailing monomial} $\TM_{\ord}(f)$
and the \emph{trailing term} $\TT_{\prec}(f)$ of a Laurent polynomial $f$.

\medskip
\n We will simply write $\mdeg$,  $\LC$, $\LM$, or $\LT$ (resp. $\TC$, $\TM$, or $\TT$) when there is no ambiguity.
}
\end{definition}

\begin{definition}{\rm

~\begin{itemize}


\item [{\rm (1)}] A nontrivial  ring $\R$ is \emph{local}  if for every element $x\in \R$, either~$x$ or~$1+x$ is invertible.

\item [{\rm (2)}] A  ring $\R$ is a \emph{valuation ring}\footnote{Here we follow Kaplansky's definition:
$\R$ may have nonzero zero-divisors.}  if  every two elements are comparable
 w.r.t.\ division (with explicit divisibilty test), i.e.\ if, given $a,b \in \R$,
either there exists $ c \in \R$ such that $b=ac$ or  there exists $ d \in \R$ such that $a=bd$ ($\R$ may have nonzero zero-divisors).
A valuation ring is a local ring. A typical example of a valuation domain which is not a field is
$\ZZ_{p\ZZ}:=\{\frac{a}{b} \in \QQ \mid a \in \ZZ \;{\rm and}\; b \in \ZZ \setminus p\mathbb{Z}\}$,
where $p$ is a prime number. A typical example of a valuation ring with nonzero zero-divisors is
$\ZZ/p^{\alpha}\ZZ$, where $p$ is a prime number and $\alpha \geq 2$.

\item [{\rm (3)}] A ring $\R$ is \emph{arithmetical} if  it is locally a valuation ring.
This is equivalent to the fact that it satisfies the following property:

\smallskip \noindent
\begin{equation} \label{eqari}
\forall x,y\in\R\quad \exists \, t,a,b \in\R\;\left\{
\begin{array}{rcl}
        (1-t)\,x & = & a\,y    \\
        b\,x & = & t\,y\\
\end{array}
\right.\quad \quad \quad \quad
\end{equation}

\n Thus, $x$ divides $y$ in the ring $\R_t:=\R[\frac{1}{t}]$ and $y$ divides $x$
in the ring $\R_{1-t}:=\R[\frac{1}{1-t}]$ (we see here that, locally, $\R$ behaves like a valuation ring).

\smallskip

\n An integral domain is called a \emph{Pr\"ufer domain}\index{Pr\"ufer domain} if it is
arithmetical.

\smallskip

\n A ring $\R$ is \emph{B\'ezout} if each finitely generated ideal is principal. It is well-known that
an integral domain $\R$ is B\'ezout if and only if it is both a Pr\"ufer and a gcd domain.
\end{itemize}
}
\end{definition}



\begin{definition} \label{def222Spaire} {\rm \cite{Y5,Y6}  Let $\R$ be a  valuation domain,
 and ${\ord}$ a
monomial order on $\Pol$.

\begin{description}
  \item [{\rm (1)}] Let $f, g \in \R[X_1,\ldots, X_n] \setminus \{ 0 \}$.
If $\mdeg(f)= \alpha$ and ${\rm
mdeg}(g)=\beta$ then let $\gamma=(\gamma_1,\ldots,\gamma_n)$, where

\n $\gamma_i= \max (\alpha_i,\beta_i)$ for each $i$.
Perform the test $\;$ $\LC(f) \mid \LC(g)\;$ or $\;$ $\LC(g) \mid \LC(f)$.


$S_{\ord}(f,g)= \frac{X^\gamma}{\LM(f)} f-\frac{{\rm
LC}(f)}{\LC(g)}\frac{X^\gamma}{\LM(g)}g$ $\;$ if $\;$
$\LC(g)$ $\;$ divides $\;$ $\LC(f)$.

 $S_{\ord}(f,g)= \frac{{\rm
LC}(g)}{\LC(f)} \frac{X^\gamma}{\LM(f)}
f-\frac{X^\gamma}{\LM(g)}g$ $\;$ if $\;$ $\LC(f)$ $\;$
divides $\;$ $\LC(g)$ and $\;$ $\LC(g)$ $\;$ does not
divide $\;$ $\LC(f)$.

\n $S_{\ord}(f,g)$ is called the {\it $S$-polynomial}
of $f$ and $g$. It is ``designed" to produce cancellation of leading
terms. We convene that $S_{\ord}(0,h)=S_{\ord}(h,0)=0$ for any  $h \in \R[X_1,\ldots, X_n]$.

\n We will simply write $S(f,g)$ when there is no ambiguity.
\end{description}
}
\end{definition}

The following definition was given in \cite{LNY} for coherent rings. We will specify it for valuation domains.

\begin{definition}\label{defdef2} {\rm Let $\R$ be a valuation domain and $\ord$~a monomial order on~$\Pol$. Let \(f_1,\dots,f_p\in \Pol\).
The \emph{$S_{\ord}$-set} of $(f_1,\dots,f_p)$ is


$$
S_{\ord}(f_1,\dots,f_p) := \sotq {S_{\ord}(f_i,f_j)}{1 \leq i < j \leq p} \subseteq \gen{f_1,\dots,f_p}.
$$
\n By induction, we define the {\it iterated $S_{\ord}$-sets} by
$S_{\ord}^0(f_1,\dots,f_p):=\{f_1,\dots,f_p \}$, $S_{\ord}^1(f_1,\dots,f_p):=\{f_1,\dots,f_p  \} \cup S_{\ord}(f_1,\dots,f_p)$, and
$S_{\ord}^{q+1}(f_1,\dots,f_p):=S_{\ord}^{q}(f_1,\dots,f_p)\cup S_{\ord}(S_{\ord}^{q}(f_1,\dots,f_p))$.

\n We also define $$S_{\ord}^{\infty}(f_1,\dots,f_p):= \mathop{\bigcup\big\uparrow}_{q \in \NN}\, S_{\ord}^q(f_1,\dots,f_p).$$


\n We will simply write $S^{q}(f_1,\dots,f_p)$ when there is no ambiguity.

}
\end{definition}

\medskip

The following theorem was given in \cite{LNY} for coherent rings. We will specify it for valuation domains.
For a proof, the reader can refer to the proof of Theorem 10 in \cite{BAY} (it is about the trailing terms ideal
but can easily be adapted to the leading terms ideal).

\begin{theorem} \label{tPrelim}
  \label{trailing}  Let $\R$ be a  valuation domain and ${\ord}$~a monomial order on~$\Pol$.  Let \(f_1,\dots,f_p\in \Pol\)
not all zero. For any \(h\in\gen{f_1,\dots,f_p}\) there exists
  $q \in \mathbb{N}$ and $g_1,\ldots,g_{t} \in S^{q}(f_1,\dots,f_p)$ such that $ \LT(h)\in \gen{\LT(g_1),\dots,\LT(g_{t})}$.
In other words,
$$\LT(\gen{f_1,\dots,f_p})= \gen{\LT(S^{\infty}(f_1,\dots,f_p))}.$$
\end{theorem}

\subsection{The case of rational monomial orders }\label{s31}

We now give four easy but important lemmas. For
matrix $M = (\alpha_{i,j}) \in \MatZ$, we define the
ring endomorphism
\[
  \phi_M\mbox{:}\ \LPol \to \LPol,\,  X_i \mapsto \prod_{j=1}^n X_j^{\alpha_{j,i}},
\]

\n which restricts to an endomorphism
\[
  \phi_M\mbox{:}\ \Pol \to \Pol
\]

\n whenever $M \in\MatZp$.

\begin{lemma} \label{LemPhiM}%
  Let $\R$ be a  ring and~$\ord$ be a rational monomial order on
  $\LPol$, given by a matrix
  $M \in \MatZ$. Then for
  $f \in \LPol \setminus \{ 0 \}$, we have
  \[
  \phi_M(\LT_{\ord}(f)) = \LTl\bigl(\phi_M(f)\bigr).
  \]
\end{lemma}

\begin{proof}{Proof.} This is because~$\phi_M$ is strictly increasing on monomials.
\end{proof}

\begin{notation} \label{serr}{\rm We denote by $\tlex$ the lexicographic monomial order with
$X_1 < X_2 < \cdots < X_n$. Let $\R$ be a ring and consider a monomial order $\ord$ on $\Pol$. The set
$$S_{\ord}:=\{ f \in \Pol \; \mid \; \LC_{\ord}(f)=1 \}$$

\n of $\ord$-monic polynomials  is a multiplicative subset of the ring $\Pol$. We set

$$\SCR:=S_{{\ord}}^{-1}\Pol.$$

\n In case $n=1$, this ring is the same as the classical ``Serre's conjecture ring"  $\SRX$ (recall that in the univariate case the only
monomial order is the natural one).}
\end{notation}

\begin{lemma} \label{LemPhiL}
  Let $\R$ be a ring and~$\ord$ be a rational monomial order on $\LPol$, given by a rank $n$ matrix
  $M = (\alpha_{i, j})\in \MatZ$. Let $k \in \NN_{> 0}$ be such that $L := k M^{-1}\in \MatZ$ and write $ L = (\beta_{i, j})$.
Let $\A : = \SCR$ and $\B := \SCRl$.
Then the maps $\left\{ X_i \mapsto \prod_{j=1}^n X_j^{\alpha_{j,i}} \right\}$ and
$\left\{ X_i \mapsto \prod_{j=1}^n X_j^{\beta_{j,i}}\right\}$ induce injective ring homomorphisms
$\phi_M:\A \to \B$ and $\phi_L:\B \to \A$ which satisfy
$\phi_M \circ \phi_L = \phi_{k \cdot \Id_n}$ and $\phi_L \circ \phi_M = \phi_{k \cdot \Id_n}$.
\end{lemma}

\begin{proof}{Proof.}
Since $S_{\ord} \cap S_{\tlex}$ contains the monomials of $\Pol$, we have $\A = S_{\ord}^{-1}\LPol$ and $\B = S_{\tlex}^{-1}\LPol$.
As the homomorphism $\phi_M$ of Lemma \ref{LemPhiM} satisfies $\phi_M(S_{\ord}) \subseteq S_{\tlex}$, it induces a homomorphism $A \rightarrow B$
which we also denote, abusively, by $\phi_M$. The remaining assertions are readily checked.
\end{proof}

\begin{lemma} \label{Spol}%
 Let $\R$ be a valuation domain, and ${\ord}$ a rational monomial order on
$\Pol$ given by a rank $n$ matrix $M \in \MatZp$.
Then for any $f,g \in \Pol \setminus \{ 0 \}$, there exists a monomial ${\cal N} \in \Pol$ such
that $$\phi_M(S_{\ord}(f,g))= {\cal N} \cdot S_{\tlex}( \phi_M(f), \phi_M(g)).$$

\n More precisely, ${\cal N}= \frac{
\phi_M\left(\lcm(\LM_{\ord}(f),\,\LM_{\ord}(g))\right)}{
\lcm\left(\phi_M(\LM_{\ord}(f)),\,\phi_M(\LM_{\ord}(g))\right)}$.
\end{lemma}

\begin{proof}{Proof.}
Denote  $\LC_{\ord}(f)=a$,  $\LC_{\ord}(g)=b$,
${\cal M}=\lcm(\LM_{\ord}(f),\LM_{\ord}(g))$, ${\cal N}_1=\frac{{\cal M}}{\LM_{\ord}(f)}$,
${\cal N}_2=\frac{{\cal M}}{\LM_{\ord}(g)}$,  and suppose that $b$  divides $a$.
We have $S_{\ord}(f,g)= {\cal N}_1 \cdot f-\frac{a}{b}{\cal N}_2 \cdot g$, and thus
$\phi_M(S_{\ord}(f,g))= \phi_M({\cal N}_1) \cdot \phi_M(f)-
\frac{a}{b}\phi_M({\cal N}_2) \cdot \phi_M(g)$.
Now, as ${\cal N}_1 \cdot \LM_{\ord}(f)= {\cal N}_2 \cdot \LM_{\ord}(g)$,
we have $\phi_M({\cal N}_1) \cdot \phi_M(\LM_{\ord}(f))= \phi_M({\cal N}_2) \cdot \phi_M(\LM_{\ord}(g))$,
and thus, by virtue of Lemma \ref{LemPhiM}, $\phi_M({\cal N}_1) \cdot \LMl\bigl(\phi_M(f)\bigr) = \phi_M({\cal N}_2) \cdot \LMl\bigl(\phi_M(g)\bigr) =\phi_M({\cal M})$.
Since $\phi_M({\cal M})$ is a common multiple of
$\LMl\bigl(\phi_M(f)\bigr)$ and $\LMl\bigl(\phi_M(g)\bigr)$,
it is divisible by $\lcm(\LMl\bigl(\phi_M(f)\bigr),\LMl\bigl(\phi_M(g)\bigr))=:{\cal M}'$.
Setting ${\cal N}:= \frac{\phi_M({\cal M})}{{\cal M}'}$, we have  $ \phi_M(S_{\ord}(f,g))= {\cal N} \cdot S_{\tlex}( \phi_M(f), \phi_M(g))$.
\end{proof}

\begin{example}{\rm Let $f=Y+aX^2,\, g=X+bY^2 \in \R[X,Y]$, where $\R$ is a valuation domain and $a,b \in \R$ such that $a$ divides $b$.
Let us consider the graded lexicographic monomial order {\ttfamily  grlex} with $X>Y$.
Its corresponding matrix is $M= \left(\begin{smallmatrix} 1 & 1 \\ 1 &
        0 \end{smallmatrix}\right)$. We have:

        \medskip

        $S_{{\rm grlex}}(f,g)=\frac{b}{a}Y^2f-X^2g=\frac{b}{a}Y^3-X^3$,

        $\phi_M(X)=XY,\, \phi_M(Y)=X,\, \phi_M(f)=aX^2Y^2+X,\, \phi_M(g)=bX^2+XY$,

        $S_{\tlex}( \phi_M(f), \phi_M(g))=\frac{b}{a}\phi_M(f)-Y^2\phi_M(g)=\frac{b}{a}X-XY^3$, and finally

        $ \phi_M(S_{{\rm grlex}}(f,g))=\frac{b}{a}X^3-X^3Y^3= X^2 \cdot S_{\tlex}( \phi_M(f), \phi_M(g))$.

}
\end{example}

\begin{lemma} \label{Spolmon}%
 Let $\R$ be a valuation domain, and ${\ord}$   a  monomial order  on $\Pol$.
Then for any $f, g \in \Pol \setminus \{ 0 \}$ and monomials ${\cal M}_1,{\cal M}_2 \in \Pol$,
there exists a monomial ${\cal N} \in \Pol$ such that
$$S({\cal M}_1\cdot f,{\cal M}_2 \cdot g)= {\cal N} \cdot S(f,g).$$

 \n More precisely, ${\cal N}= \frac{\lcm({\cal M}_1\cdot  \LM(f),\,{\cal M}_2 \cdot \LM(g))}{\lcm(\LM(f),\,\LM(g))}$.
\end{lemma}

\begin{proposition} \label{Sinfty}%
  Let $\R$ be a valuation domain,  $\ord$   a rational monomial order  on $\Pol$ given by
 a rank~$n$ matrix $M \in \MatZp$, and \(f_1,\dots,f_p\in \Pol\). Then:

\begin{itemize}
\item [{\rm (1)}]  For any $q \in \mathbb{N} $ and  $g \in S_{\ord}^{q}(f_1,\dots,f_p)$,
there exist a monomial ${\cal N} \in \Pol$ and $h \in S_{\tlex}^{q}( \phi_M(f_1),\dots, \phi_M(f_p))  $  such
 that $ \phi_M(g)= {\cal N} \cdot h$.
\item [{\rm (2)}] For any $q \in \mathbb{N} $ and $h \in S_{\tlex}^{q}( \phi_M(f_1),\dots, \phi_M(f_p))  $,
there exist a monomial ${\cal N} \in \Pol$ and $g \in S_{\ord}^{q}(f_1,\dots,f_p)$  such
 that $ \phi_M(g)= {\cal N} \cdot h$.

\item [{\rm (3)}] The ideal $\gen{f_1,\dots,f_p}$ contains
a $\ord$-monic polynomial if and only if  the ideal $\gen{\phi_M(f_1),\dots, \phi_M(f_p)}$ contains a $\tlex$-monic polynomial.

\end{itemize}

\end{proposition}

\begin{proof}{Proof.} Let us prove (1), the proof of (2) being similar. We induct on $q$. For $q=0$, this is clear as we can take $ {\cal N}=1$.
Suppose that the result is true for $q$ and let us prove it for $q+1$. Let $g \in S_{\ord}^{q+1}(f_1,\dots,f_p)$.
If $g \in S_{\ord}^{q}(f_1,\dots,f_p)$ then we are done by the induction hypothesis since
$ S_{\tlex}^{q}( \phi_M(f_1),\dots, \phi_M(f_p)) \subseteq S_{\tlex}^{q+1}( \phi_M(f_1),\dots, \phi_M(f_p))$.
Else, $g=S_{\ord}(g_1,g_2)$ with $g_1,g_2 \in S_{\ord}^{q}(f_1,\dots,f_p)$.
By the induction hypothesis, there exist  monomials ${\cal N}_1,{\cal N}_2 \in \Pol$ and $h_1,h_2 \in S_{\tlex}^{q}( \phi_M(f_1),\dots, \phi_M(f_p))$
such  that $\phi_M(g_1)= {\cal N}_1 \cdot h_1$ and $ \phi_M(g_2) = {\cal N}_2 \cdot h_2$.
By Lemma \ref{Spol}, we have $ \phi_M(g)= \phi_M(S_{\ord}(g_1,g_2))= {\cal N}_3 \cdot S_{\tlex}( \phi_M(g_1), \phi_M(g_2))$ for some monomial ${\cal N}_3 \in \Pol$.
By Lemma \ref{Spolmon}, we obtain that
$\phi_M(g)= {\cal N}_3 \cdot S_{\tlex}( {\cal N}_1 \cdot h_1, {\cal N}_2 \cdot h_2) = {\cal N}_3 \cdot{\cal N}_4 \cdot S_{\tlex}( h_1, h_2)$
for some  monomial ${\cal N}_4 \in \Pol$. The desired result follows since
$S_{\tlex}( h_1, h_2) \in S_{\tlex}^{q+1}( \phi_M(f_1),\dots, \phi_M(f_p))$.

 \smallskip

 \n (3) If the ideal $\gen{f_1,\dots,f_p}$ contains a $\ord$-monic polynomial $P$, then
$\phi_M(P) \in \gen{\phi_M(f_1),\dots, \phi_M(f_p)}$ is a $\tlex$-monic polynomial by Lemma \ref{LemPhiM} (this holds for any ring $\R$).
If the ideal $\gen{\phi_M(f_1),\dots, \phi_M(f_p)}$ contains a $\ord$-monic polynomial, then $\gen{f_1,\dots,f_p}$ contains a
$\tlex$-monic polynomial by Theorem \ref{tPrelim} and assertion (2) above.
\end{proof}

\begin{corollary} \label{CorSInfty}%
  Let $\R$ be a Pr\"ufer domain and let $\ord$ be a rational monomial order  on $\Pol$ given by
 a rank~$n$ matrix $M \in \MatZp$. Let \(f_1,\dots,f_p\in \Pol\). The ideal $\gen{f_1,\dots,f_p}$ contains a $\ord$-monic polynomial
if and only if  the ideal $\gen{\phi_M(f_1),\dots, \phi_M(f_p)}$ contains a $\tlex$-monic polynomial.
\end{corollary}

\begin{proof}{Proof.} The property  that a given ideal $I$ of $\Pol$ contains a monic polynomial is a local-global, that is,  it suffices
to check it locally. A local Pr\"ufer domain is a valuation domain. The desired result follows from Proposition \ref{Sinfty}(2).
\end{proof}

\begin{proposition} \label{monic}
  Let $\R$ be a valuation domain of Krull dimension at most $1$, let ${\ord}$ be a rational monomial order on $\Pol$, and $I$ a nonzero finitely
 generated ideal of   $\Pol$. Then there exists an element
$a \in \R$  such that $\LT_{\ord}(I)= a \cdot J$, where $J$ is an ideal of $\Pol$  which contains a monomial. In particular, $\LC_{\ord}(I)=\gen{a}$ is principal.
\end{proposition}

\begin{proof}{Proof.} Write $I= \gen{f_1,\dots,f_p}$ where $f_i \in \Pol \setminus \{ 0 \}$, and suppose that
${\ord}$  is given by
 a rank~$n$ matrix $M = (\alpha_{i,j}) \in \mathbb{Z}_{\ge 0}^{n \times n}$. As in the proof of  \cite[Theorem~4]{Y4},
one can write $\gen{\phi_M(f_1),\dots, \phi_M(f_p)} = b \cdot \Delta \cdot  \gen{g_1,\ldots,g_s}$, where $b \in \R$, $\Delta, \,g_j \in \Pol \setminus \{ 0 \}$ and
$\gen{g_1,\ldots,g_s}$ contains a polynomial $g$ whose leading coefficient is $1$ with respect to the monomial order $\tlex$.
The desired result follows since $\LTl(\gen{\phi_M(f_1),\dots, \phi_M(f_p)}) = b \cdot \LTl(\Delta) \cdot  \LTl( \gen{g_1,\ldots,g_s})$,
and by virtue of Proposition \ref{Sinfty} ($a= b \cdot \LCl(\Delta)$).
\end{proof}

While the counterexample given in \cite{Y-JA21} dispels the Gr\"obner ring conjecture  for the irrational case, Proposition \ref{monic}
above bolsters the evidence for this conjecture in the rational case since in order to show that for a  nonzero finitely
 generated ideal $I$  of   $\Pol$ ($\R$  being a one-dimensional valuation domain) and a rational monomial
 order $\ord$ on $\Pol$, the ideal $\LT_{\ord}(I)$  is finitely generated, one can suppose w.l.o.g that $I$ contains a
$\ord$-monic polynomial exactly as in the lexicographic case \cite[Proof of Theorem 4]{Y4}.


\section{The ring $\R_{ {\ord}}\gen{X_1,\ldots,X_n}$ }\label{s333}

 The proof of the  Gr\"obner Ring Conjecture in the  univariate case \cite{LSY} relies heavily on Brewer-Costa theorem \cite[Theorem 1]{BC}:
\begin{theorem}{\em (Brewer-Costa)} \label{ThBC}
 Let $\R$ be a ring.
\begin{itemize}
\item [{\rm (1)}]   $\R\gen{X}$ is a B\'ezout domain  if and only if ~$\R$ is a
B\'ezout domain of Krull dimension $\leq 1$. In this case, $\R\gen{X}$ has Krull dimension $\leq 1$.

\item [{\rm (2)}]  $\R\gen{X}$ is a Pr\"ufer domain  if and only if ~$\R$ is a
Pr\"ufer domain of Krull dimension $\leq 1$. In this case, $\R\gen{X}$ has Krull dimension $\leq 1$.
\end{itemize}
\end{theorem}

Also, the cornerstone of
the proof of the  Gr\"obner Ring Conjecture in the lexicographic order case \cite{Y4} is
the fact that the ring $\SCRl$ (see Notation \ref{serr})
is a B\'ezout domain of Krull dimension $\leq 1$  if so is  $\R$.
This is a consequence of the univariate case since the ring $\SCRl$
(often denoted by $\R\gen{X_1,\ldots,X_n}$ or simply $\R\gen{n}$  in the literature)
can also be defined  by induction as
$$\SCRl = (\R_{\tlex}\gen{X_1,\ldots,X_{n-1}})\gen{X_n}.$$

Thus, we have established, as an immediate consequence of Theorem \ref{ThBC}:

\begin{proposition} \label{PropLex}
Let $\R$ be a ring.
Then the following are equivalent:
\begin{itemize}
\item[$(1)$] $\R$ is a B\'ezout (resp. Pr\"ufer) domain of Krull dimension at most $1$.
\item[$(2)$] $\SCRl$ is a B\'ezout (resp. Pr\"ufer) domain.
\end{itemize}
Furthermore, if any of the above assertions holds, then the Krull dimension of $\SCRl$ is at most $1$.
\end{proposition}

Let us reassure the reader by recalling that non-Noetherian B\'ezout domains of Krull dimension $1$ do exist. 
The domain of all algebraic integers is one example \cite[Theorem 102]{Kap}. 
Other examples can be built by means of the Krull-Jaffard-Ohm Theorem \cite[Theorem III.5.3]{FS}.

Our goal in this section is to generalize, in a constructive way,
Theorem \ref{ThBC} to $\R_{{\ord}}\gen{X_1,\ldots,X_n}$ where ${\ord}$ is an arbitrary rational monomial order, which is done with Theorem \ref{ThBCRational} below.
We shall first give a complete constructive proof of Theorem \ref{ThBC}. (Note that the constructive proof given in \cite{LQY} deals only with the if part.)

\begin{notation} \label{cond} {\rm
\n If $I$ is an ideal of a ring $\R$ and $b \in \R$, we define

$$[I:a^{\infty}]:=\{x \in \R  \mid  \exists \, n \in \mathbb{N} \mid
x\, a^n \in I \}.$$

\n For $a,\,b \in \R$, we define

\vspace{-1em}

$$[b:a^{\infty}]:= [\gen{b}:a^{\infty}].$$}
\end{notation}

\begin{lemma} \label{imonic-dim1} Let $\R$ be a ring.
For any $a,\,b \in \R$, we have
$\LT(\gen{1+aX,b})= [b:a^{\infty}][X] + \gen{aX}$.
\end{lemma}

\begin{proof}{Proof.} The reader can consult the solution of Exercise 387 in \cite{Y5}.
\end{proof}


\begin{proof}{Proof of Theorem \ref{ThBC}}  (1) ``$\Leftarrow$" See \cite{LQY}.
\smallskip

\n ``$\Rightarrow$" Let us first prove that $\R$ has Krull dimension $\leq 1$.
Pick  $a,\,b \in \R \setminus \{0\}$. Since ${\rm gcd}(1+aX,\,b)=1$ and $\R\gen{X}$ is a B\'ezout domain,
the ideal $\gen{1+aX,b}$ of $\R[X]$  contains a monic polynomial. By virtue of Lemma \ref{imonic-dim1},
we have $1 \in [b:a^{\infty}] + \gen{a}$ and thus there exist $\alpha \in \R$ and $x\in [b:a^{\infty}] $  such that $1=x + \alpha \,a$.
It follows that there exists $n \in \mathbb{N}$ such that
$(1- \alpha \,a)a^n \in \gen{b}$, as desired. Since $\R$ is a gcd domain (see Lemma \ref{LemGCD})  of Krull dimension $\leq 1$,
it is a B\'ezout domain   by \cite[Theorem~XI.3.12]{LQ}.

\smallskip \n (2) As being a Pr\"ufer domain and having Krull dimension $\leq 1$ are local-global properties,
we can suppose that $\R$ is local. A local Pr\"ufer domain is a valuation domain and thus a B\'ezout domain.
The result follows from (1).
\end{proof}
\medskip

Hereafter the main result of this paper.

\begin{theorem} \label{ThBCRational}
 Let $\R$ be a ring, and  ${\ord}$   a rational monomial order  on $\Pol$.
 \begin{itemize}
\item [{\rm (1)}]  $\R_{{\ord}}\gen{X_1,\ldots,X_n}$ is a B\'ezout domain  if and only if ~$\R$ is a
B\'ezout domain of Krull dimension $\leq 1$. In this case, $\R_{{\ord}}\gen{X_1,\ldots,X_n}$ has Krull dimension $\leq 1$.

\item [{\rm (2)}] $\R_{{\ord}}\gen{X_1,\ldots,X_n}$ is a Pr\"ufer domain  if and only if ~$\R$ is a
Pr\"ufer domain of Krull dimension $\leq 1$.
In this case, $\R_{{\ord}}\gen{X_1,\ldots,X_n}$ has Krull dimension $\leq 1$.
\end{itemize}
\end{theorem}

We first address the part of Theorem \ref{ThBCRational} which pertains to Krull dimension by establishing the following:

\begin{proposition}\label{PropKrullDim}  Let $\R$ be a ring. Let $\ord$ be a rational monomial order on $\Pol$. Then we have:
$$\dim \SCR =\dim \SCRl\; \stackrel{{\small {\rm \cite{CEK}}}}{ =} \; \dim \Pol - n.$$
\end{proposition}
Note that $\dim \SCR = \infty$ obviously holds if $\dim \R = \infty$.
Here it is worth recalling that the equality $\dim \Pol = \dim \R + n$ is not always true.
Nevertheless, it holds if $\R$ is Noetherian or Pr\"ufer.
Let us also observe that Proposition \ref{PropKrullDim} does not hold if $\ord$ is an irrational order, see Example \ref{CounterExample} below.
\medskip

The next results prepare the proof of Proposition \ref{PropKrullDim}.

\begin{lemma} \label{LemIndependence}
Let $\R$ be a ring. Let $n, k \in \NN_{> 0}$ and let $\Rk = \{0, \dots, k - 1\}^n$.
Then the ring $\Pol$ is freely generated as an $\Polk$-module
by the set of monomials $\left\{X^{\alpha} \, \vert \, \alpha \in \Rk \right\}$.
\end{lemma}

\begin{proof}{Proof.}
Let $f = \sum_{\beta} a_{\beta} X^{\beta} \in \Pol$, with every $a_{\beta}$ lying in $\R$.
Then $f = \sum_{\alpha \in \mathcal{R}_k} b_{\alpha}(f) X^{\alpha}$ where $b_{\alpha}(f) \in \Polk$
is uniquely defined by
$b_{\alpha}(f) := \frac{1}{X^\alpha} \sum_{\beta \equiv \alpha \,\mod \,k \ZZ^n} a_{\beta} X^{\beta}$.
Conversely, if $f = \sum_{\alpha \in \mathcal{R}_k} c_{\alpha} X^{\alpha}$ for some $c_{\alpha} \in \Polk$,
then it is immediate to check that $c_{\alpha} = b_{\alpha}(f)$ for every $\alpha \in \Rk$. Therefore
$\sum_{\alpha \in \mathcal{R}_k} c_{\alpha} X^{\alpha} = 0$, if and only if, $c_{\alpha} = 0$ for every $\alpha$.
\end{proof}

\begin{proposition} \label{PropQ}
Let $\R$ be a ring and let $\ord$ be a monomial order on $\A := \Pol$.
Let $k \in \NN_{> 0}$, $f_1, \dots, f_p \in \A$ and set
$\tilde{f_i} := \phi_{k \cdot \Id_n}(f_i)$.
Assume there is $P \in \PolY$ such that $\TCl(P)$ is $\ord$-monic (resp. $\TCl(P) = 1$) and
$P(\tilde{f_1}, \dots, \tilde{f}_p) = 0$.
Then there is $Q \in \PolY$ such that
$\TCl(Q)$ is $\ord$-monic (resp. $\TCl(Q) = 1$) and $Q(f_1, \dots, f_p) = 0$.
\end{proposition}

\begin{proof}{Proof.}
Writing $P = \sum_{\beta} c_{\beta} Y^{\beta}$ and
$c_{\beta} = \sum_{\alpha \in \Rk} c_{ \alpha \beta} X^{\alpha}$
with $\Rk$ as in Lemma \ref{LemIndependence} and $c_{\alpha \beta} \in \Polk$,
the identity $P(\tilde{f_1}, \dots, \tilde{f}_m) = 0$ can be reworded as
\begin{equation} \label{EqCalphaBeta}
\sum_{\alpha \in \Rk} \left(\sum_{\beta} c_{\alpha \beta} \tilde{f}^{\beta} \right) X^{\alpha} = 0.
\end{equation}

\n Let $\beta_{\min}$ be the multidegree of the trailing term of $P$ with respect to $\tlex$.
Since by hypothesis the polynomial $c_{\beta_{\min}}$ is $\ord$-monic (resp. $c_{\beta_{\min}} = 1$),
there is a multidegrees $\alpha_0 \in \Rk$ such that $c_{\alpha_0 \beta_{\min}}$ is
$\ord$-monic (resp. $c_{\alpha\beta_{\min}} = 1$).
Let $\tilde{Q} := \sum_{\beta} c_{\alpha_0 \beta} Y^{\beta}$.
It follows from (\ref{EqCalphaBeta}) and Lemma \ref{LemIndependence} that
$\tilde{Q}(\tilde{f_1}, \dots, \tilde{f}_p) = 0$, hence the result.
\end{proof}

\begin{proposition} \label{PropQScr}
Let $\R$ be a ring.
Let $\A := \SCR$ and $p, k \in \NN_{> 0}$. Let $f_1, \dots, f_p \in \A$ and set
$\tilde{f_i} := \phi_{k \cdot \Id_n}(f_i)$ for $i \in \{1, \dots, p\}$.
Assume there is $P \in \PolY$ such that $\TCl(P) = 1$ and
$P(\tilde{f_1}, \dots, \tilde{f}_m) = 0$.
Then there is $Q \in \PolY$ such that $\TCl(Q) = 1$ and $Q(f_1, \dots, f_p) = 0$.
\end{proposition}

\begin{proof}{Proof.}
Let $s_1, \dots, s_p$ be $\ord$-monic polynomials such that $g_i := s_i f_i \in \Pol$.
It is easy to check that the following are equivalent:

\begin{itemize}
\item
There is $Q \in \PolY$ such that $\TCl(Q) = 1$ and $Q(f_1, \dots, f_p) = 0$.
\item
There is $Q \in \Pol[Y_1, \dots, Y_p]$ such that $\TCl(Q)$ is $\prec$-monic and $Q(g_1, \dots, g_p) = 0$.
\end{itemize}
Thus the result follows from Proposition \ref{PropQ}.
\end{proof}

\begin{definition}{\rm
Let $\A$ be a ring. A sequence $(y_1, \dots, y_p)$ of elements of
$\A$ is said to be $\tlex$\emph{-dependent} if there is
$P \in \PolY$ such that $\TCl(P) = 1$ and $P(y_1, \dots, y_p) = 0$.}
\end{definition}

\begin{proposition} \label{PropKrullDimIndep} {\em \cite[Proposition XIII.2.8]{LQ}}
Let $\A$ be a ring and let $p \in \NN_{> 0}$. Then the Krull dimension
of $\A$ is less than $p$ if and only if every sequence of $p$ elements in $\A$ is $\tlex$-dependent.
\end{proposition}

We are now in position to prove Proposition \ref{PropKrullDim}.

\begin{proof}{Proof of Proposition \ref{PropKrullDim}}
Set $\A :=\SCR, \B := \SCRl$ and let $p \in \NN_{>0}$.
By Proposition \ref{PropKrullDimIndep}, it suffices to show that every sequence of $p$ elements in
$\A$ is $\tlex$-dependent, if and only if, if every sequence of $p$ elements in $\B$ is $\tlex$-dependent.
Let $M \in \MatZp$ be a rank $n$ defining matrix of $\ord$ and let
$k \in \NN_{> 0}$ and $L$ be as in Lemma \ref{LemPhiL}.

\smallskip \n Let $f_1, \dots, f_p \in \A$. Let $P \in \B[Y_1, \dots, Y_p]$ and write
$P= \sum_{\beta} c_{\beta}  Y^{\beta}$. Assume that $\TCl(P) = 1$ and
$P(\phi_M(f_1), \dots, \phi_M(f_p)) = 0$.
Applying $\phi_L$ on both sides of the latter identity, we obtain $\phi_L(P)(\tilde{f_1}, \dots, \tilde{f_p}) = 0$
where $\phi_L(P) :=  \sum_{\beta} \phi_L(c_{\beta}) Y^{\beta}$ and $\tilde{f_i} = \phi_{k \cdot \Id_n}(f_i)$.
Clearly, we have $\TCl(\phi_L(P)) = 1$.
By Proposition \ref{PropQScr}, we can find $Q \in \A[Y_1, \dots, Y_p]$ such that $\TCl(Q) = 1$
and $Q(f_1, \dots, f_p) = 0$. Therefore the sequence $(f_1, \dots, f_p)$ is $\tlex$-dependent.

\smallskip \n Consider now $f_1, \dots, f_p \in \B$. Let $P \in \A[Y_1, \dots, Y_p]$ be such that $\TCl(P) = 1$ and
$P(\phi_L(f_1), \dots, \phi_L(f_p)) = 0$.
Applying $\phi_M$ on both sides of the latter identity, we obtain $\phi_M(P)(\tilde{f_1}, \dots, \tilde{f_p}) = 0$.
By Proposition \ref{PropQScr}, we can find $Q \in \B[Y_1, \dots, Y_p]$ such that $\TCl(Q) = 1$
and $Q(f_1, \dots, f_p) = 0$. Hence $(f_1, \dots, f_p)$ is $\tlex$-dependent.
\end{proof}

Further results are required to prove Theorem \ref{ThBCRational}.

\begin{lemma} \label{LemGcdAndBezout}
Let $\R$ be a ring and let $\ord$ be a monomial order on $\Pol$.
Let $k \in \NN_{>0}$ and $f,g \in \R[X_1, \dots, X_n]$.
Set $\tilde{f} = \phi_{k \cdot \Id_n}(f)$ and $\tilde{g} = \phi_{k \cdot \Id_n}(g)$.
Then the following hold:

\begin{itemize}
\item[$(1)$] $\langle f, g \rangle = \Pol$ if and only if
$\langle \tilde{f}, \tilde{g} \rangle = \Pol$.
\item[$(2)$]
$\langle f, g \rangle = \SCR$ if and only if
$\langle \tilde{f}, \tilde{g} \rangle = \SCR$.
\end{itemize}

If $\R$ is an integral domain, then the following also hold:
\begin{itemize}
\item[$(3)$] If $f$ and $g$ have no non-constant divisor, then neither have $\tilde{f}$ and $\tilde{g}$.
\item[$(4)$] $\gcd(f, g) = 1$ if and only if $\gcd(\tilde{f}, \tilde{g}) = 1$.
\end{itemize}
\end{lemma}

\begin{proof}{Proof.}
$(1)$. The proof of  the ``only if" part is trivial.
Let $f, g \in \Pol$ such that $p \tilde{f} + q \tilde{g} = 1$ for some $p, q \in \Pol$.
Setting $P(Y_1, Y_2) = pY_1 + q Y_2 - 1$ and using Proposition \ref{PropQ} yields the result.

\smallskip

\n $(2)$ Reason as in $(1)$, using Proposition \ref{PropQScr}.

\smallskip

\n $(3)$ Assume first that $n=1$ and let $\K$ be the quotient field of  of $\R$. Since $\K[X_1]$ is a principal ideal domain,
we can find $p, q\in \R[X_1]$ and $r\in \R$ such that $pf + qg = r$ and hence
$\tilde{p}\tilde{f} + \tilde{q} \tilde{g} = r$.
Thus $\tilde{f}$ and $\tilde{g}$ have no common factor of positive degree.
Suppose now that $n >1$.
Resorting to the case $n = 1$, we see that a common divisor of
$\tilde{f}$ and $\tilde{g}$ has degree zero with respect to $X_i$ for every $i$.
Hence it has total degree zero.

\smallskip

\n $(4)$ The result follows by virtue of $(3)$ and
the fact that a polynomial of degree zero divides $f$ in $\Pol$
if and only if it divides $\tilde{f}$.
\end{proof}

\begin{lemma} \label{LemGCD}
Let $\R$ be a ring and let $\ord$ be a monomial order on $\Pol$.
Then the following are equivalent:
\begin{itemize}
\item[$(i)$] $\R$ is a gcd domain.
\item[$(ii)$] $\SCR$ is a  gcd domain.
\end{itemize}
\end{lemma}

\begin{proof}{Proof.}
$(i) \Rightarrow (ii)$. The ring $\Pol$ is a gcd domain by \cite[Theorem XI.3.16]{LQ} and so is $\SCR$
because the gcd property is stable under taking localization \cite[Fact XI.3.6]{LQ}.

\smallskip \n $(ii) \Rightarrow (i)$. 
Let $\A := \SCR$ and let $a, b \in \R$. 
Assume that $\gcd_{\A}(a, b) = \frac{f}{s}$ holds in $\A$ with $f, s \in \Pol$ and $\LC(s) = 1$. 
Let $c := \LC(f)$. 
We shall prove that  $\gcd_{\R}(a, b) = c$ holds in $\R$. Let $d$ be a common divisor of $a$ and $b$ in $\R$.
Since $d$ divides $\frac{f}{s}$ in $\SCR$, we can find $f', s' \in \Pol$ such that $\LC(s') = 1$ and $\frac{f}{s} = d \frac{f'}{s'}$.
Using the identity $\LC(s'f) = \LC(sdf')$, we infer that $d$ divides $c$ in $\R$.
We can show in the same way that $c$ divides both $a$ and $b$ in $\R$, which proves the result.
\end{proof}

\begin{lemma} \label{LemBezout}
Let $\R$ be gcd domain and let $\ord$ be a rational monomial order on $\Pol$ given by a rank $n$ matrix $M \in \MatZp$.
Let $f, g \in \Pol$.
Then the following are equivalent:
\begin{itemize}
\item[$(1)$] $\gcd(f, g) = 1$ holds in $\SCR$.
\item[$(2)$] $\gcd(\phi_M(f), \phi_M(g)) = 1$ holds in $\SCRl$.
\end{itemize}
\end{lemma}

\begin{proof}{Proof.}
Because of Lemma \ref{LemPhiM}, we can assume, without loss of generality that $f, g \in \Pol$.
Let $k \in \NN_{> 0}$ and $L$ be as in Lemma \ref{LemPhiL}.
We set $\A := \SCR$ and $\B := \SCRl$.

\smallskip \n $(1) \Rightarrow (2)$. Let $h \in \B$ be a common divisor of $\phi_M(f)$ and $\phi_M(g)$ in $\B$.
As we aim at proving that $h$ is a unit of $\B$, we can suppose, without loss of generality that $h \in \Pol$.
Let $\tilde{f} = \phi_{k \cdot \Id_n}(f)$ and $\tilde{g} = \phi_{k \cdot \Id_n}(g)$.
Since $\phi_L \circ \phi_M(f) = \tilde{f}$ and
$\phi_L \circ \phi_M(g) = \tilde{g}$, we deduce that $\phi_L(h)$
divides both $\tilde{f}$ and $\tilde{g}$ in the Laurent polynomial ring $\LPol$.
By Lemma \ref{LemGcdAndBezout}(4), we know that $\gcd(\tilde{f}, \tilde{g}) = 1$ holds in $\Pol$.
Thus we have $\gcd(\tilde{f}, \tilde{g}) = 1$ in $\LPol$ by \cite[Fact XI.3.6]{LQ}.
As a result, $\phi_L(h)$ is a unit of $\LPol$, which implies that
$h$ is made of a unique term whose coefficient is invertible. Therefore
$\gcd(\phi_M(f), \phi_M(g)) = 1$ holds in $\B$.

\smallskip \n $(2) \Rightarrow (1)$. Let $h \in \A$ be a common divisor of $f$ and $g$ in $\A$.
We can suppose, without loss of generality, that $h \in \Pol$. Then $\phi_M(h)$ is a common divisor of
$\phi_M(f)$ and $\phi_M(g)$ in $\B$. Therefore $\phi_M(uh)$ a $\tlex$-monic polynomial for some unit $u \in \R$.
It follows from Lemma 6 that $uh$ is a $\ord$-monic polynomial, which proves that $\gcd(f, g) = 1$ holds in $\A$.
\end{proof}

\begin{proposition} \label{PropBCRational}
Let $\R$ be a ring and let $\ord$ be a rational monomial order on $\Pol$.
Then the following are equivalent:
\begin{itemize}
\item[$(1)$] $\SCR$ is a B\'ezout domain.
\item[$(2)$] $\SCRl$ is a B\'ezout domain.
\end{itemize}
\end{proposition}

\begin{proof}{Proof.}
Let $M \in \MatZp$ be a rank $n$ matrix defining $\ord$.
Let $k \in \NN_{> 0}$ and $L$ be as in Lemma \ref{LemPhiL} and
set $\A :=\SCR$, $\B := \SCRl$.

\smallskip \n $(1) \Rightarrow (2)$
By Lemma \ref{LemGCD}, the ring $\B$ is a gcd domain. 
Hence, in order to show that $\B$ is B\'ezout, it suffices to show
that for every pair $f, g \in \B$ satisfying $\gcd(f, g) = 1$ in $\B$,
we have $\langle f, g \rangle_{\B} = \B$.
We may assume, without loss of generality, that $f, g \in \Pol$ and that $\gcd(f, g) = 1$ holds in $\Pol$.
Let $\tilde{f} = \phi_{k \cdot \Id_n}(f)$ and $\tilde{g} = \phi_{k \cdot \Id_n}(g)$.
By Lemma \ref{LemGcdAndBezout}$(4)$, we know that
$\gcd(\tilde{f}, \tilde{g}) = 1$ holds in $\Pol$ and hence in $\A$.
Thus $\gcd(\phi_L(f), \phi_L(g)) = 1$ holds in $\A$ by Lemma \ref{LemBezout}.
Since $\A$ is a B\'ezout domain by hypothesis, the identity
$\langle \phi_L(f), \phi_L(g) \rangle_{\A} = \A$ is satisfied.
Thus $1 = \phi_M(1) \in \langle \tilde{f}, \tilde{g} \rangle_{\B}$ so that
$\langle \tilde{f}, \tilde{g} \rangle_{\B} = \B$.
Applying Lemma \ref{LemGcdAndBezout}$(2)$ completes the proof.

\smallskip \n $(2) \Rightarrow (1)$
By Lemma \ref{LemGCD}, we know that $\A$ is a gcd domain. 
We consider $f, g \in \A$ such that $\gcd(f, g) = 1$ holds in $\A$, 
setting ourselves to show that $\langle f, g \rangle_{\A} = \A$.
We may assume, without loss of generality, that $f,g \in \Pol$.
Thanks to Corollary \ref{CorSInfty}, we only need to check that
$\langle \phi_M(f), \phi_M(g) \rangle_{\B} = \B$.
As $\B$ is a B\'ezout domain by hypothesis,
the latter identity follows from Lemma \ref{LemBezout}.
\end{proof}

The following corollary is obtained by means of a standard localization argument.
\begin{corollary} \label{CorPrufer}
Let $\R$ be a ring and let $\ord$ be a rational monomial order on $\Pol$.
Then the following are equivalent:
\begin{itemize}
\item[$(i)$] $\SCR$ is a Pr\"ufer domain.
\item[$(ii)$] $\SCRl$ is a Pr\"ufer domain.
\end{itemize}
\end{corollary}

\begin{proof}{Proof of Theorem \ref{ThBCRational}}
Combine Propositions \ref{PropLex}, \ref{PropKrullDim}, \ref{PropBCRational} and Corollary \ref{CorPrufer}.
\end{proof}

Note that Theorem \ref{ThBCRational} is no longer true when
the considered monomial order is irrational as shows  the following example.

\begin{example}\label{CounterExample} {\rm Let's take again the counterexample given in \cite{Y-JA21} and consider
  a valuation domain $\V$ whose valuation group is $G= (\mathbb{Z} + \mathbb{Z} \sqrt{2},\;+)$. Denote by ${\mathsf v}$
the valuation corresponding to $\V$.  Since $G$ is a subgroup of $(\mathbb{R},\;+)$, it is archimedean and thus $\dim \V = 1$. Let us pick $a,\;b \in \V$ such that  ${\mathsf v}(a)= 1$ and ${\mathsf v}(b)=\sqrt{2}$, and consider the  ideal $J=\gen{-1+aX,\; -1+bY}$ of $\V[X,Y]$. Let us equip $\V[X,Y]$ with  the irrational monomial order ${{\ord}}$  given by the row matrix $ (1  \;\; \sqrt{2})$.
 We know by Claim 2 of the counterexample given in \cite{Y-JA21} that $J \cap S_{\ord} =\emptyset$
and thus there is no B\'ezout identity between $-1+aX$ and $-1+bY$ in $\V_{\ord}\gen{X,Y}$ despite that ${\rm gcd}(-1+aX,\; -1+bY)=1$. We conclude that $\V_{\ord}\gen{X,Y}$ is not a B\'ezout domain contrary to $\V_{\tlex}\gen{X,Y}$. Since $\V_{\ord}\gen{X,Y}$ is a gcd domain and knowing that a Pr\"ufer gcd domain  is a B\'ezout domain as well as a gcd domain   of Krull dimension $\leq 1$ \cite[Theorem~XI.3.12]{LQ}, we infer that $\V_{\ord}\gen{X,Y}$ is not a Pr\"ufer domain and its Krull dimension is not $\leq 1$.
}
\end{example}


\begin{thebibliography}{}
\bibitem{BAY} Faten Ben Amor, Ihsen Yengui. {\it The trailing terms ideal}. J. Algebra Appl. Vol. {\bf 20}, No. 9 (2021).

\bibitem{Bi67}  Errett Bishop. {\it Foundations of constructive analysis}. McGraw-Hill, New York, 1967.

\bibitem{BC} James W. Brewer, Douglas L. Costa.
{\it  Projective modules over some non-Noetherian polynomial rings}.
J. Pure Appl. Algebra {\bf 13} (1978) 157--163.

\bibitem{CEK} Paul-Jean Cahen,  Zahra Elkhayyari,  Salah-Eddine Kabba. {\it Krull and valuative dimension
of the Serre conjecture ring $R  \langle n \rangle $}. Lect. Not.
Pure and Appl. Math. {\bf 185} (1997) 173--185.

\bibitem{FS}
L\'azl\'o~Fuchs, Luigi~Salce.
\newblock {\em Modules over non-{N}oetherian domains}, volume~84 of {\em
  Mathematical Surveys and Monographs}.
\newblock American Mathematical Society, Providence, RI, 2001.

\bibitem{Kap}
Irving Kaplansky.
\newblock Commutative rings.
\newblock Boston: {Allyn} and {Bacon}, {Inc}. {X}, 180 p. (1970).

\bibitem{Kemper:Trun:VanAnh:2017} Gregor Kemper, Ngo Viet Trung, Nguyen Thi Van Anh. {\it Toward a theory of monomial preorders}.
  Math. Comp. {\bf 87} (2018) 2513--2537.

 \bibitem{KY} Gregor Kemper, Ihsen Yengui. {\it Valuative dimension and monomial orders}.
J. Algebra {\bf 557} (2020) 278--288.

\bibitem{Lom98a} Henri Lombardi.
{\it Dimension de Krull, Nullstellens\"atze et \'Evaluation
dynamique}. Math. Zeitschrift   {\bf 242} (2002) 23--46.

\bibitem{LNY} Henri Lombardi, Stefan Neuwirth,  Ihsen Yengui. {\it The syzygy theorem for strongly discrete coherent rings},  Preprint 2019.

\bibitem{LQ}
Henri Lombardi, Claude Quitt{\'e}. {\it Commutative algebra: constructive methods. Finite projective
  modules}. Algebra and Applications, 20, Springer, Dordrecht, 2015. Translated from the French (Calvage \& Mounet, 2011, revised and
  extended by the authors) by Tania K. Roblot.

\bibitem{LQY} Henri Lombardi, Claude Quitt\'e, Ihsen Yengui.
{\it Hidden constructions in abstract algebra (6) The theorem of
Maroscia, Brewer and Costa}. J. Pure Appl. Algebra {\bf 212} (2008)
1575--1582.

\bibitem{LSY} Henri Lombardi, Peter Schuster, Ihsen Yengui. {\it The Gr\"obner
ring conjecture in one variable}. Math. Z. {\bf  270} (2012)
1181--1185.

\bibitem{MRR} Ray Mines, Fred Richman,  Wim Ruitenburg. {\it A course in constructive algebra. Universitext}. Springer, New York, 1988.

\bibitem{MoY} Samiha Monceur, Ihsen Yengui. {\it On the leading terms ideals
of polynomial ideals over a valuation ring}. J. Algebra {\bf 351}
(2012) 382--389.

\bibitem{Y4} Ihsen Yengui. {\it  The Gr\"obner Ring Conjecture in the lexicographic order case}. Math. Z. {\bf
276} (2014) 261--265.

\bibitem{Y5} Ihsen Yengui. {\it Constructive Commutative Algebra}. Lecture Notes in Mathematics, no 2138, Springer 2015.

\bibitem{Y-JA21} Ihsen Yengui.  {\it A counterexample to the Gr\"obner ring conjecture}. J. Algebra  {\bf 586} (2021) 526--536

\bibitem{Y6} Ihsen Yengui. {\it Computational Algebra: Course and Exercises with Solutions}. World Scientific Publishing 2021.
\end{thebibliography}
\end{document}